\documentclass[12pt]{amsart}
\usepackage{mathrsfs}
\usepackage{amsmath}
\usepackage{amsfonts}
\usepackage{amssymb}
\usepackage[hidelinks]{hyperref}
\usepackage{subfigure}
\usepackage{color}
\usepackage{fancyhdr}
\usepackage{bbm}
\usepackage{bm}
\usepackage{graphicx}
\usepackage{cite}
\newtheorem{theorem}{Theorem}[section]
\newtheorem{remark}{Remark}[section]

\newtheorem{lemma}{Lemma}[section]
\numberwithin{equation}{section}

\begin{document}

\title{On the Sign Changes of a Weighted Divisor Problem}

\author{Lirui Jia}
\address{School of Mathematical Sciences, Zhejiang University,
Hangzhou 310027, People's Republic of China }

\email{jialirui@126.com}

\author{Tianxin Cai}
\address{School of Mathematical Sciences, Zhejiang University,
Hangzhou 310027, People's Republic of China }

\email{txcai@zju.edu.cn}

\author{Wenguang Zhai}
\address{Department of Mathematics, China University of Mining and Thechnology, Beijing 100083, People's Republic of China }

\email{zhaiwg@hotmail.com}

\thanks{
The first and the second author are supported by the National Natural Science Foundation
of China (Grant No. 11571303). The third author is supported by the National Key Basic Research Program of China (Grant No. 2013CB834201), the National Natural Science Foundation of China (Grant No. 11171344).}
\date{}{}


\begin{abstract}
Let $S\big(x; \frac{a_1}{q_1}, \frac{a_2}{q_2}\big)=\mathop{{\sum}'}_{mn\leq x} \cos\big(2\pi m\frac{a_1}{q_1}\big)\sin\big(2\pi n\frac{a_2}{q_2}\big)$ with $x\geq q_1q_2, 1\leq a_i\leq q_i$, and $(a_i, q_i)=1$ ($i=1, 2$). We study the sign changes of $S\big(x; \frac{a_1}{q_1}, \frac{a_2}{q_2}\big)$, and prove that for a  sufficiently large constant $C$, $S\big(x; \frac{a_1}{q_1}, \frac{a_2}{q_2}\big)$  changes sign in the interval $[T,T+C\sqrt{T}]$ for any large $T$. Meanwhile, we show that for a small constant $c'$,  there exist infinitely many
subintervals of length $c'\sqrt{T}\log^{-7}T$ in $[T,2T]$ where $\pm S\big(t; \frac{a_1}{q_1}, \frac{a_2}{q_2}\big)> c_5 (q_1q_2)^\frac{3}{4}t^\frac{1}{4}$ always holds.
\end{abstract}

\maketitle


\section{Introduction}

\subsection{Dirichlet divisor problem}

Let $d(n)$ be the Dirichlet divisor function, and $D(x)=\sum\limits_{n\leq x}d(n)=\sum\limits_{n_1n_2\leq x}1$ denote the summatory function. In 1849, Dirichlet proved that
$$D(x)=x\log x+(2\gamma-1)x+O(\sqrt{x}),$$
where $\gamma$ is the Euler constant.

 Let
$$\Delta(x)=D(x)-x\log x-(2\gamma-1)x$$ be the error term in the asymptotic formula for $D(x)$. Dirichlet's divisor problem consists of determining the smallest $\alpha$, for which $\Delta(x)\ll x^{\alpha+\varepsilon}$ holds for any $\varepsilon>0$. Clearly, Dirichlet's result implies that $\alpha\leq\frac{1}{2}$. Since then, there are many improvements on this estimate.
 The best to-date is given by
Huxley\cite{Huxley03,huxley2005exponential}, reads
\begin{equation}\label{huxley}
\Delta(x)\ll x^\frac{131}{416}\log^\frac{26947}{8320}x.
\end{equation}
It is widely conjectured that α$\alpha=\frac{1}{4}$ is admissible and is the best possible.

Since $\Delta(x)$ exhibits considerable fluctuations, one natural way to study the upper bounds is to consider the moments.

In 1904, Voronoi \cite{voronoi1904fonction} showed that
\begin{equation*}
  \int_1^T\Delta(x)dx=\frac{T}{4}+O(T^{\frac{3}{4}}).
\end{equation*}
Later, in 1922 Cram\'er\cite{cramer1922zwei} proved the mean square formula
\begin{equation*}
   \int_1^T\Delta(x)^2dx=cT^{\frac{3}{2}}+O(T^{\frac{5}{4}+\varepsilon}),\quad\forall~\varepsilon>0,
\end{equation*}
where $c$ is a positive constant. In 1983, Ivic \cite{ivic1983large} used the method of large values to prove that
\begin{equation}\label{ivic}
    \int_1^T|\Delta(x)|^Adx\ll T^{1+\frac{A}{4}+\varepsilon},\quad\forall~\varepsilon>0
\end{equation}
for each fixed $0\leq A\leq\frac{35}{4}$. The range of $A$ can be extended to $\frac{262}{27}$ by the estimate \eqref{huxley}.
In 1992, Tsang\cite{tsang1992higher}  obtained the asymptotic formula
\begin{equation}\label{01}
  \int_1^T\Delta(x)^kdx=c_kT^{1+\frac{k}{4}}+O(T^{1+\frac{k}{4}-\delta_k}),\quad \text{for}\ k=3,4,
\end{equation}
with positive constants $c_3$, $c_4$, and $\delta_3=\frac{1}{14}$, $\delta_4=\frac{1}{23}$. Ivi\'c and Sargos~\cite{ivic2007higher} improved the values $\delta_3$, $\delta_4$ to $\delta'_3=\frac{7}{20}$, $\delta'_4=\frac{1}{12}$, respectively.
Heath-Brown\cite{heath1992distribution} in 1992 proved that for any integer $k<A$, where $A$ satisfies \eqref{ivic}, the limit
$$c_k=\lim_{X\rightarrow\infty}X^{-1-\frac{k}{4}}\int_1^X\Delta(x)^kdx$$
exists.  Then, there followed a series
of investigations on explicit asymptotic formula of the type \eqref{01} for larger values of $k$. In 2004, Zhai \cite{zhai04} established asymptotic formulas for $3\leq k\leq9$.

 At the beginning of the 20th  century, Voronoi\cite{voronoi1904fonction} proved the remarkable exact formula that
  \begin{equation*}
    \Delta(x)=-\frac{2}{\pi}\sqrt{x}\sum_{n=1}^\infty\frac{d(n)}{\sqrt{n}}\big(K_1(4\pi\sqrt{nx})\big)+ \frac{\pi}{2}Y_1(4\pi\sqrt{nx}),
  \end{equation*}
  where $K_1$, $Y_1$ are the Bessel functions, and the series on the right-hand side is boundedly convergent for $x$ lying in each fixed closed interval.

Heath-Brown and Tsang \cite{heathbrown1994sign} studied the sign changes of $\Delta(x)$. They proved that for a suitable constant $C > 0$, $\Delta(x)$ changes sign on the interval $[T,T + C \sqrt{T}]$ for every sufficiently large $T$. Here the length $\sqrt{T}$ is almost best possible since they proved that in the interval $[T,2T]$ there are many subintervals of length $\gg\sqrt{T}\log^{-5}T$ such that $\Delta(x)$ does not change sign in any of these subintervals.

\subsection{A weighted divisor problem}

Recently, Berndt et al\cite{Berndt06, Berndt12} considered a weighted divisor function $\mathop{{\sum}'}_{mn\leq x}\cos(2\pi m\theta_1)\sin(2\pi n\theta_2)$, and got an analogue of Voronoi's formula as follows.

Let $J_1$ be the ordinary Bessel function. If $0<\theta_1, \theta_2<1$ and $x>0$,  then
\begin{align*}
   &\mathop{{\sum}'}\limits_{mn\leq x}\cos(2\pi m\theta_1)\sin(2\pi n\theta_2) \\
   =&  -\frac{\cot(\pi\theta_2)}{4}+\frac{\sqrt{x}}{4}\sum_{m=0}^\infty\sum_{n=0}^\infty \Bigg\{\frac{J_1\Big(4\pi\sqrt{(m+\theta_1)(n+\theta_2)x}\Big)}{\sqrt{(m+\theta_1)(n+\theta_2)}}\\
   &\!+\frac{J_1\Big(4\pi\sqrt{(m+1-\theta_1)(n+\theta_2)x}\Big)}{\sqrt{(m+1-\theta_1)(n+\theta_2)}} -\frac{J_1\Big(4\pi\sqrt{(m+\theta_1)(n+1-\theta_2)x}\Big)}{\sqrt{(m+\theta_1)(n+1-\theta_2)}}\\ &-\frac{J_1\Big(4\pi\sqrt{(m+1-\theta_1)(n+1-\theta_2)x}\Big)}{\sqrt{(m+1-\theta_1)(n+1-\theta_2)}}\Bigg\}.
\end{align*}

Denote
$$S\big(x; \frac{a_1}{q_1}, \frac{a_2}{q_2}\big)=\mathop{{\sum}'}_{mn\leq x} \cos\Big(2\pi m\frac{a_1}{q_1}\Big)\sin\Big(2\pi n\frac{a_2}{q_2}\Big).$$
 In \cite{jia2016weighted}, we got for $x\geq q_1q_2, 1\leq a_i\leq q_i$, $(a_i, q_i)=1$ ($i=1, 2$) that
\begin{equation*}
  S\big(q_1q_2x; \frac{a_1}{q_1}, \frac{a_2}{q_2}\big)\ll q_1q_2x^\frac{131}{416}\left(\log x\right)^\frac{26947}{8320},
\end{equation*}
\begin{equation}\label{s1}
   \int_1^T{S\big(q_1q_2x; \frac{a_1}{q_1}, \frac{a_2}{q_2}\big)}dx\ll q_1q_2T^{\frac{3}{4}}.
\end{equation}
If $T\gg (q_1q_2)^{\varepsilon}$ is large enough, then for $2\leq k\leq9$ we proved
\begin{equation}\label{sk}
  \int_1^{T}{S^k\big(q_1q_2x; \frac{a_1}{q_1}, \frac{a_2}{q_2}\big)}dx
  =(q_1q_2)^kC_k\int_1^Tx^{\frac{k}{4}}dx+o\big((q_1q_2)^{k} T^{1+\frac{k}{4}}\big),
\end{equation}
where $C_k$ are  explicit constants.

Here we study $S\big(x; \frac{a_1}{q_1}, \frac{a_2}{q_2}\big)$ further and give some more estimates about it.

\textsc{Notations}. For a real number $t$, let $[t]$ be the largest integer no greater than $t$, $\{t\}=t-[t]$, $\psi(t)=\{t\}-\frac{1}{2}$, $\parallel t\parallel=\min(\{t\}$, $1-\{t\})$, $e(t)=e^{2\pi it}$. $\mathbb{C}$, $\mathbb{R}$, $\mathbb{Z}$, $\mathbb{N}$ denote the set of complex numbers, of real numbers, of integers, and of natural numbers, respectively; $f\asymp g$ means that both $f\ll g$ and $f\gg g$ hold. Throughout this paper, $\varepsilon$ denote sufficiently small positive constants, and $\mathcal{L}$ denotes $\log T$.

\vspace{1ex}

\section{Main results}\label{sec:results}

In this paper, we will discuss the sign changes of $S\big(x; \frac{a_1}{q_1}, \frac{a_2}{q_2}\big)$ and get the following
\begin{theorem}\label{th:change}
 Let $c_1 > 0$ be a sufficiently small constant and $c_2 > 0$ be a sufficiently
large constant, $q_1\geq 2$, $q_2\geq 3$, $1\leq a_i\leq q_i$ and $(a_i, q_i)=1$ $(i=1, 2)$. For any real-valued function $|f(t)|\leq c_1 t^\frac{1}{4}$, the
function $S\big(t; \frac{a_1}{q_1}, \frac{a_2}{q_2}\big)+f(t)$  changes sign at least once in the interval $[T,T + c_2 \sqrt{q_1q_2T}]$ for every sufficiently large $T\geq (q_1q_2)^{1+\varepsilon}$. In particular, there exist $t_1$, $t_2\in [T,T +c_2 \sqrt{q_1q_2T}]$ such that $S\big(t_1; \frac{a_1}{q_1}, \frac{a_2}{q_2}\big)\geq c_1 t_1^\frac{1}{4}$ and $S\big(t_2; \frac{a_1}{q_1}, \frac{a_2}{q_2}\big)\leq -c_1 t_2^\frac{1}{4}$.
\end{theorem}

\begin{theorem}\label{th:maintain}
There exist positive absolute constants $c_3$ ,$c_4$ ,$c_5 $ such that, for
any large parameter $T\geq (q_1q_2)^{1+\varepsilon}$, there are at least $c_3 \sqrt{T} log^{7}T$ disjoint subintervals of length $c_4 \sqrt{T} log^{-7}T$ in $[T,2T]$, such that $\pm S\big(t; \frac{a_1}{q_1}, \frac{a_2}{q_2}\big)> c_5 (q_1q_2)^\frac{3}{4}t^\frac{1}{4}$, whenever $t$ lies in any of these
subintervals. Moreover, we have the estimate
\begin{equation*}
    meas\big\{t\in[T,2T]:\pm S\big(t; \frac{a_1}{q_1}, \frac{a_2}{q_2}\big)> c_5 (q_1q_2)^\frac{3}{4}t^\frac{1}{4}\big\}\gg T.
\end{equation*}
\end{theorem}

We also study the $\Omega$-result of the error term in the asymptotic formula \eqref{sk} for odd $k$ by using Theorem \ref{th:maintain}. Define
\begin{equation*}
    \mathcal{F}_k\big(q_1q_2x; \frac{a_1}{q_1}, \frac{a_2}{q_2}\big):=(q_1q_2)^{-k}\int_1^{T}{S^k\big(q_1q_2x; \frac{a_1}{q_1}, \frac{a_2}{q_2}\big)}dx
  -C_kT^{1+\frac{k}{4}}.
\end{equation*}
We have the following
\begin{theorem}\label{th:omega}
The estimate
\begin{equation*}
    \mathcal{F}_k\big(q_1q_2T; \frac{a_1}{q_1}, \frac{a_2}{q_2}\big)=\Omega\big(T^{\frac{1}{2}+\frac{k}{4}}\mathcal{L}^{-7}\big)
\end{equation*}
holds for any fixed odd integer $k\geq3$ and every sufficiently large $T\geq (q_1q_2)^{\varepsilon}$.
\end{theorem}
\begin{remark}
Although at the present moment we can only prove \eqref{sk} for $2\leq k\leq 9$, Theorem \ref{th:omega} holds for any odd $k\geq 2$.
\end{remark}
\begin{remark} We can get the same or similar conclusions with all presented here for $\mathop{{\sum}'}_{mn\leq x}{\sin(2\pi n\frac{a_1}{q_1})\sin(2\pi m\frac{a_2}{q_2})}$ and $\mathop{{\sum}'}_{mn\leq x}{\cos(2\pi n\frac{a_1}{q_1})\cos(2\pi m\frac{a_2}{q_2})}$ with the same approach.
\end{remark}
\vspace{1.5ex}

\section{The Voronoi-type formula for $S\big(x; \frac{a_1}{q_1}, \frac{a_2}{q_2}\big)$}

In \cite{jia2016weighted},  we proved an analogue of Voronoi's formula for $S\big(q_1q_2x; \frac{a_1}{q_1}, \frac{a_2}{q_2}\big)$.

Denote
\begin{align*}
\Delta d_2(n;a_1, q_1, a_2, q_2)
=&d(n;  a_1, q_1, a_2, q_2)+d(n; - a_1, q_1, a_2, q_2)\\
&-d(n;  a_1, q_1, -a_2, q_2)-d(n; - a_1, q_1, a_2, q_2),
\end{align*}
\begin{align*}
&\Delta d_{2,1}(n,H,J;a_1, q_1, a_2, q_2)\\
=&\!\!\!\!\mathop{{\sum}'}_{\begin{subarray}{c}n=hl\\1\leq h\leq H\\h\leq l\leq2^{J+1}h\\ h\equiv a_2\!\!\!\!\!\pmod{q_2}\\l\equiv  a_1\!\!\!\!\!\pmod{q_1}\end{subarray}}\!\!\!1+\!\!\!\!\mathop{{\sum}'}_{\begin{subarray}{c}n=hl\\1\leq h\leq H\\h\leq l\leq2^{J+1}h\\ h\equiv a_2\!\!\!\!\!\pmod{q_2}\\l\equiv - a_1\!\!\!\!\!\pmod{q_1}\end{subarray}}\!\!\!1-\!\!\!\mathop{{\sum}'}_{\begin{subarray}{c}n=hl\\1\leq h\leq H\\h\leq l\leq2^{J+1}h\\ h\equiv -a_2\!\!\!\!\!\pmod{q_2}\\l\equiv a_1\!\!\!\!\!\pmod{q_1}\end{subarray}}\!\!\!1-\!\!\!\mathop{{\sum}'}_{\begin{subarray}{c}n=hl\\1\leq h\leq H\\h\leq l\leq2^{J+1}h\\ h\equiv -a_2\!\!\!\!\!\pmod{q_2}\\l\equiv - a_1\!\!\!\!\!\pmod{q_1}\end{subarray}}\!\!\!1,
\end{align*}
\begin{align*}
&\Delta d_{2,2}(n,H,J;a_1, q_1, a_2, q_2)\\
=&\!\!\!\!\mathop{{\sum}'}_{\begin{subarray}{c}n=hl\\1\leq h\leq H\\h\leq l\leq2^{J+1}h\\ h\equiv a_2\!\!\!\!\!\pmod{q_2}\\l\equiv  a_1\!\!\!\!\!\pmod{q_1}\end{subarray}}\!\!\!1+\!\!\!\!\mathop{{\sum}'}_{\begin{subarray}{c}n=hl\\1\leq h\leq H\\h\leq l\leq2^{J+1}h\\ l\equiv a_2\!\!\!\!\!\pmod{q_2}\\h\equiv - a_1\!\!\!\!\!\pmod{q_1}\end{subarray}}\!\!\!1-\!\!\!\mathop{{\sum}'}_{\begin{subarray}{c}n=hl\\1\leq h\leq H\\h\leq l\leq2^{J+1}h\\ l\equiv -a_2\!\!\!\!\!\pmod{q_2}\\h\equiv a_1\!\!\!\!\!\pmod{q_1}\end{subarray}}\!\!\!1-\!\!\!\mathop{{\sum}'}_{\begin{subarray}{c}n=hl\\1\leq h\leq H\\h\leq l\leq2^{J+1}h\\ l\equiv -a_2\!\!\!\!\!\pmod{q_2}\\h\equiv - a_1\!\!\!\!\!\pmod{q_1}\end{subarray}}\!\!\!1.
\end{align*}

Let $J=[\frac{\mathcal{L}+2\log q_1q_2-4\log \mathcal{L}}{\log2}]$, $H\geq2$ be a parameter to be determined, and $ T^\varepsilon<y\leq \min(H^2, (q_1q_2)^2 T)\mathcal{L}^{-4}$. Suppose $T\leq x\leq 2T$. Then
\begin{align}\label{vor0}
   S\big(q_1q_2x; \frac{a_1}{q_1}, \frac{a_2}{q_2}\big)
=&R_0(x; y)\!+\!R_{12}(x; y, H)\!+\!R_{21}(x; y, H)\\
\nonumber&+\!G_{12}(x; H)\!+\!G_{21}(x; H)
  \!+\!O(q_1q_2\mathcal{L}^3),
\end{align}
where
 \begin{gather}
    \label{R_0}
    R_0(x; y)=\frac{q_1q_2x^\frac{1}{4}}{4\sqrt{2}\pi}\sum_{n\leq y}\frac{\cos\left(4\pi\sqrt{nx}-\frac{3\pi}{4}\right)}{n^{\frac{3}{4}}}\Delta d_2(n;a_1, q_1, a_2, q_2),\\
    \label{16}
    \nonumber R_{12}(x; y, H)\!=\!\frac{q_1q_2x^\frac{1}{4}}{4\sqrt{2}\pi}\!\!\sum_{y<n\leq2^{J+1}H^2 }\!\!\frac{\cos\big(4\pi\sqrt{nx}\!-\!\frac{3\pi}{4}\big)}{n^{\frac{3}{4}}}\Delta d_{2,1}(n,H,J;a_1, q_1, a_2, q_2),\\
  \nonumber R_{21}(x; y, H)\!= \!\frac{q_1q_2x^\frac{1}{4}}{4\sqrt{2}\pi}\!\!\sum_{y<n\leq 2^{J+1}H^2 }\!\!\frac{\cos\big(4\pi\sqrt{nx}\!-\!\frac{3\pi}{4}\big)}{n^{\frac{3}{4}}}\Delta d_{2,2}(n,H,J;a_1, q_1, a_2, q_2),\\
    \nonumber  G_{12}(x; H) = O\Bigl(q_2\sum_{\begin{subarray}{c}  n_1\leq q_1\sqrt{T}\end{subarray}}\min\Bigl(1, \frac{1}{H\|\frac{q_1x}{n_1}-\frac{r_2}{q_2}\|}\Bigr)\Bigr),\label{G12}\\
 \nonumber  G_{21}(x; H)=O\Big(q_1\sum_{\begin{subarray}{c}  n_2\leq q_2\sqrt{T}\end{subarray}}\min\Big(1, \frac{1}{H\|\frac{q_2x}{n_2}-\frac{r_1}{q_1}\|}\Big)\Big).
   \end{gather}
\section{Proof of Theorem \ref{th:change}}

In this section, we prove Theorem \ref{th:change} following the approach of \cite{heathbrown1994sign}.

Let $n_0$ be the smallest integer $n$, such that $\Delta d_2(n;a_1, q_1, a_2, q_2)\!\neq\!0$. By the definition of $\Delta d_2(n;a_1, q_1, a_2, q_2)$, it is easy to see that $\Delta d_2(n_0;a_1, q_1, a_2, q_2)\!=\!1$ or $\!-1$, and  $n_0\!=\!\min\{a_1,q_1\!-a_1\}\!\times\min\{a_2,q_2\!-a_2\}$, which suggests $n_0\!<\!\frac{1}{4}q_1q_2$.

Suppose  $|f(t)|\leq c_1 t^\frac{1}{4}$. Let
\begin{equation*}
    S^*(t)=4\sqrt{2}\pi(q_1q_2)^{-1}t^{-\frac{1}{2}}\Big(S\big(q_1q_2t^2; \frac{a_1}{q_1}, \frac{a_2}{q_2}\big)+f(q_1q_2t^2)\Big),\quad \text{for }t\geq1.
\end{equation*}
Define
\begin{equation*}
    K_\zeta(u):=(1-|u|)\big(1+\zeta\sin(4\pi\alpha\sqrt{n_0}u)\big)\quad\text{for }|u|\leq1,
\end{equation*}
with $\zeta=1$ or $-1$, and $\alpha>n_0^\frac{1}{2}$ a large number.

Set $\zeta'=-\Delta d_2(n_0;a_1, q_1, a_2, q_2)\zeta$. Then it is easy to see that $\zeta'=1$ or $-1$,
and Theorem \ref{th:change} follows from Lemma \ref{lem:change} below.
\begin{lemma}\label{lem:change}
Suppose $T\gg(q_1q_2)^\varepsilon$ is a  large parameter. Then for each $\sqrt{T}\leq t\leq\sqrt{2T}$, we have
\begin{align*}
&\int_{\!-1}^1\!\!\!S^*\!(t\!+\!\alpha u)K_\zeta(u)du\!=\!\!\frac{\zeta'}{2n_0^\frac{3}{4}}\!\sin(4\pi t\sqrt{n_0}\!-\!\frac{3}{4}\pi)\!+\!O\big(\alpha^{\!-\!2}\!+\!t^{-\!\frac{1}{2}}\mathcal{L}^3 \!+\!c_1(q_1q_2)^{-\!\frac{3}{4}}\big).
\end{align*}
\end{lemma}
\begin{proof}
Let $J=[\frac{\mathcal{L}+2\log q_1 q_2-4\log \mathcal{L}}{\log2}]$, $H\geq2$ be a parameter to be determined, and $ T^\varepsilon<y\leq \min(H^2, (q_1q_2)^2 T)\mathcal{L}^{-4}$. From \eqref{vor0}, we have
\begin{align}\label{vor}
   \!S^*(t)\!=&R^*_0(t; y)\!+\!\!R^*_{12}(t; y, H)\!+\!\!R^*_{21}(t; y, H)\!+\!4\sqrt{2}\pi(q_1q_2)^{\!-1}t^{-\frac{1}{2}}f(q_1q_2t^2)\!\\
\nonumber&\!+\!O\big(t^{-\frac{1}{2}}\big(G^*_{12}(t; H)\!+\!G^*_{21}(t; H)\big)\big)
  \!+\!O\big(t^{-\frac{1}{2}}\mathcal{L}^3\big),
\end{align}
where
 \begin{gather*}
    \label{15}
    R^*_0(t; y)=\sum_{n\leq y}\frac{\cos\left(4\pi t\sqrt{n}-\frac{3\pi}{4}\right)}{n^{\frac{3}{4}}}\Delta d_2(n;a_1, q_1, a_2, q_2),\\
    \label{16}
     R^*_{12}(t; y, H)=\sum_{y<n\leq2^{J+1}H^2 }\frac{\cos\big(4\pi t\sqrt{n}-\frac{3\pi}{4}\big)}{n^{\frac{3}{4}}}\Delta d_{2,1}(n,H,J;a_1, q_1, a_2, q_2),\\
  \nonumber R^*_{21}(t; y, H)= \sum_{y<n\leq 2^{J+1}H^2 }\frac{\cos\big(4\pi t\sqrt{n}-\frac{3\pi}{4}\big)}{n^{\frac{3}{4}}}\Delta d_{2,2}(n,H,J;a_1, q_1, a_2, q_2),\\
      G^*_{12}(t; H) = O\Bigl(\frac{1}{q_1}\sum_{\begin{subarray}{c}  n_1\leq q_1\sqrt{T}\end{subarray}}\min\Bigl(1, \frac{1}{H\|\frac{q_1t^2}{n_1}-\frac{r_2}{q_2}\|}\Bigr)\Bigr),\label{G12}\\
 \nonumber  G^*_{21}(t; H)=O\Big(\frac{1}{q_2}\sum_{\begin{subarray}{c}  n_2\leq q_2\sqrt{T}\end{subarray}}\min\Big(1, \frac{1}{H\|\frac{q_2t^2}{n_2}-\frac{r_1}{q_1}\|}\Big)\Big).
   \end{gather*}
Denote
\begin{align*}
    &R^*(t)\!=\!R^*_0(t; y)\!+\!R^*_{12}(t; y, H)\!+\!R^*_{21}(t; y, H),&
    &G^*(t)\!=\!G^*_{12}(t; H)\!+\!G^*_{21}(t; H).
\end{align*}
Then
\begin{equation}\label{s}
  S^*(t)=R^*(t)\!+\!4\sqrt{2}\pi(q_1q_2)^{-1}t^{-\frac{1}{2}}f(q_1q_2t^2)\!+\!O\big(t^{-\frac{1}{2}}G^*(t)\big)
  \!+\!O\big(t^{-\frac{1}{2}}\mathcal{L}^3\big).
\end{equation}

We first consider $\int_{-1}^1G^*(t+\alpha u)du$. Noting that
\begin{equation*}
    \min\Big(1,\frac{1}{H\|r\|}\Big)=\sum_{h=-\infty}^\infty a(h)e(hr)
\end{equation*}
with
\begin{align*}
    a(0)\ll H^{-1}\log H,&&a(h)\ll\min\Big(H^{-1}\log H,h^{-2}H\Big),\ h\neq0.
\end{align*}
We have
\begin{align*}
    &\int_{-1}^1G^*_{12}(t+\alpha u;H)du\\
    =&\frac{1}{q_1}\sum_{h=-\infty}^\infty a(h)\sum_{n_1\leq q_1\sqrt{T}}e\Big(\frac{hq_1t^2}{n_1}-\frac{hr_2}{q_2}\Big)\int_{-1}^1e\Big(\frac{2hq_1t\alpha u+hq_1\alpha^2u^2}{n_1}\Big)du\\
    \ll&|a(0)|\sqrt{T}+\frac{1}{q_1}\sum_{h=1}^\infty |a(h)|\sum_{n_1\leq q_1\sqrt{T}}\frac{n_1}{hq_1t\alpha}\\
    \ll&H^{-1}T^\frac{1}{2}\log H+\sum_{h=1}^HH^{-1}\log HT(ht\alpha)^{-1}+\sum_{h=H}^\infty HT(t\alpha)^{-1}h^{-3}\\
    \ll&H^{-1}T^\frac{1}{2}\log^2 H,
\end{align*}
where the first derivative test was used. This estimate remain valid with $G^*_{12}$ replaced by $G^*_{21}$, which yields
\begin{equation}\label{G}
    \int_{-1}^1G^*(t+\alpha u)du\ll H^{-1}T^\frac{1}{2}\log^2 H.
\end{equation}

Now we  estimate $\int_{-1}^1R^*(t+\alpha u)K_\zeta(u)du$. By the elementary formula
\begin{align*}
    &\cos\big(4\pi (t+\alpha u)\sqrt{n}-\frac{3\pi}{4}\big)\\
    =&\cos\big(4\pi t\sqrt{n}-\frac{3\pi}{4}\big)\cos(4\pi \alpha u\sqrt{n})-\sin\big(4\pi t\sqrt{n}-\frac{3\pi}{4}\big)\sin(4\pi \alpha u\sqrt{n}),
\end{align*}
we get
\begin{align*}
    \int_{-1}^1\cos\big(4\pi (t+\alpha u)\sqrt{n}-\frac{3\pi}{4}\big)(1-|u|)\big(1+\zeta\sin(4\pi\alpha\sqrt{n_0}u)\big)du
    =I_1+I_2,
\end{align*}
with
\begin{align*}
    I_1=&\cos\big(4\pi t\sqrt{n}-\frac{3\pi}{4}\big)\int_{-1}^1\cos(4\pi \alpha u\sqrt{n})(1-|u|)\big(1+\zeta\sin(4\pi\alpha\sqrt{n_0}u)\big)du\\
      =&\cos\big(4\pi t\sqrt{n}-\frac{3\pi}{4}\big)\int_{-1}^1\cos(4\pi \alpha u\sqrt{n})(1-|u|)du,
 \end{align*}
  \begin{align*}
    I_2\!=&\sin\big(4\pi t\sqrt{n}-\frac{3\pi}{4}\big)\int_{-1}^1\sin(4\pi \alpha u\sqrt{n})(1-|u|)\big(1+\zeta\sin(4\pi\alpha\sqrt{n_0}u)\big)du\\
    =&\zeta\sin\big(4\pi t\sqrt{n}-\frac{3\pi}{4}\big)\int_{-1}^1\sin(4\pi \alpha u\sqrt{n})(1-|u|)\sin(4\pi\alpha\sqrt{n_0}u)du\\
   =&\frac{\zeta}{2}\sin\big(4\pi t\sqrt{n}-\frac{3\pi}{4}\big)\int_{-1}^1(1-|u|)\cos\big(4\pi\alpha u(\sqrt{n}-\sqrt{n_0})\big)du\\
   &-\frac{\zeta}{2}\sin\big(4\pi t\sqrt{n}-\frac{3\pi}{4}\big)\int_{-1}^1(1-|u|)\cos\big(4\pi\alpha u(\sqrt{n}+\sqrt{n_0})\big)du.\\
\end{align*}
By using
\begin{equation*}
     \int_0^1(1-u)\cos(A u)du\ll |A|^{-2}\quad A\neq0,
\end{equation*}
we have
\begin{align*}
    I_1\ll\alpha^{-2}n^{-1},&&
    I_2=\left\{\begin{array}{ll}\frac{\zeta}{2}\sin\big(4\pi t\sqrt{n_0}\!-\!\!\frac{3\pi}{4}\big)+O(\alpha^{-2}n_0^{-1}),&n=n_0, \\O(\alpha^{-2}(\sqrt{n}-\sqrt{n_0})^{-2}),&n\neq n_0,\end{array}\right.
\end{align*}
which suggests
\begin{align*}
    &\int_{-1}^1\cos\big(4\pi (t+\alpha u)\sqrt{n}-\frac{3\pi}{4}\big)K_\zeta(u)du\\
    =&\left\{\begin{array}{ll}\frac{\zeta}{2}\sin\big(4\pi t\sqrt{n_0}\!-\!\!\frac{3\pi}{4}\big)+O(\alpha^{-2}n_0^{-1}),&n=n_0, \\O(\alpha^{-2}(\sqrt{n}-\sqrt{n_0})^{-2}),&n\neq n_0.\end{array}\right.
\end{align*}
Take $H= T$, $y= T^\frac{1}{2}$. Then clearly $n_0<y$. Thus we get
\begin{align}\label{R}
   &\!\!\!\int_{-1}^1\!R^*(t+\alpha u)K_\zeta(u)du\\
\nonumber =\!-&\frac{\zeta}{2n_0^\frac{3}{4}}\sin\!\big(4\pi t\sqrt{n_0}\!-\!\frac{3\pi}{4}\big)\Delta d_2(n_0;a_1, q_1, a_2, q_2)\!+\!O\Big(\!\sum_{n>n_0}\!\frac{\alpha^{-2}n^{-\frac{3}{4}}d(n)}{(\sqrt{n}\!-\!\!\sqrt{n_0})^{2}}\Big)\\
\nonumber =\!-&\frac{\zeta}{2n_0^\frac{3}{4}}\sin\big(4\pi t\sqrt{n_0}-\frac{3\pi}{4}\big)\Delta d_2(n_0;a_1, q_1, a_2, q_2)+O(\alpha^{-2}),
\end{align}
by using $\sum_{n>n_0}\frac{d(n)}{n^\frac{3}{4}(\sqrt{n}-\sqrt{n_0})^{2}}\ll1$.

Note that $H=T$, $t\asymp T^\frac{1}{2}$. From \eqref{s}-\eqref{R}, we see
\begin{align*}
   &\int_{-1}^1S^*(t+\alpha u)K_\zeta(u)du\\
=&-\frac{\zeta}{2n_0^\frac{3}{4}}\sin\big(4\pi t\sqrt{n_0}\!-\!\frac{3\pi}{4}\big)\Delta d_2(n_0;a_1, q_1, a_2, q_2)+O(\alpha^{-2})\\
&+O\big((q_1q_2)^{-1}t^{-\frac{1}{2}}\sup_{|u|\leq1}f(q_1q_2(t\!+\!\alpha u)^2)\big)\!+\!O\big(t^{-\frac{1}{2}}H^{-1}T^\frac{1}{2}\mathcal{L}^2\big)
  \!+\!O\big(t^{-\frac{1}{2}}\mathcal{L}^3\big)\\
  =&\frac{\zeta'}{2n_0^\frac{3}{4}}\!\sin\big(4\pi t\sqrt{n_0}\!-\!\frac{3\pi}{4}\big)\!+\!O(\alpha^{\!-2})\!+\!O\big(c_1(q_1q_2)^{\!-\frac{3}{4}}\big)\!\!+\!O(t^{-\frac{1}{2}} \mathcal{L}^3).
\end{align*}
Thus we complete the proof of Lemma \ref{lem:change}
\end{proof}

\section{ The mean value of $S\big(q_1q_2x; \frac{a_1}{q_1}, \frac{a_2}{q_2}\big)$ in short intervals}

Suppose $T\gg (q_1q_2)^{\varepsilon}$ is a large parameter, $1\leq h\leq \frac{1}{2}\sqrt{T}$. Denote $S(q_1q_2x)=S\big(q_1q_2x; \frac{a_1}{q_1}, \frac{a_2}{q_2}\big)$.  In this section we shall
estimate the integral
\begin{equation*}
    I(T,h)=\int_1^T\big(S(q_1q_2(x+h)-S(q_1q_2x)\big)^2dx,
\end{equation*}
which would play an important role in the proof of Theorem \ref{th:maintain}. This type of integral was studied for the error term in the mean square of $\zeta(\frac{1}{2}+ it)$ by Good \cite{good1977einomega},
for the error term in the Dirichlet divisor problem by Jutila \cite{jutila1984divisor} and for the error term in Weyl's law for Heisenberg manifold by Tsang and Zhai \cite{tsang2012sign}. Here we follows the approach of Tsang and Zhai \cite{tsang2012sign} and prove the following
\begin{lemma}\label{lem:meanvalue}  The estimate
\begin{equation*}
    I(T,h)\ll (q_1q_2)^{2}hT\log^3\frac{\sqrt{T}}{h}+(q_1q_2)^{2}T\mathcal{L}^6
\end{equation*}
holds  uniformly for $1\leq h\leq \frac{1}{2}\sqrt{T}$.
\end{lemma}

\begin{proof}
Write
\begin{equation}\label{I}
    I(T,h)=\int_1+\int_2,
\end{equation}
where
\begin{align*}
    \int_1=&\int_1^{100\max(h^2,T^\frac{2}{3})}(S(q_1q_2(x+h)-S(q_1q_2x)\big)^2dx,\\
    \int_2=&\int_{100\max(h^2,T^\frac{2}{3})}^T(S(q_1q_2(x+h)-S(q_1q_2x)\big)^2dx.
\end{align*}
From \eqref{sk}, we see that
\begin{equation}\label{int1}
    \int_1\ll (q_1q_2)^2(h^3+T)\ll (q_1q_2)^2Th.
\end{equation}

For $\int_2$, first we estimate the integral
\begin{equation}\label{J}
    J(U,h)\!=\!\!\int_U^{2U}\!\!(S(q_1q_2(x\!+\!h)\!-\!S(q_1q_2x)\big)^2dx, \quad100\max(h^2,T^\frac{2}{3})\!\leq\! U\!\leq\! T.
\end{equation}
Let $T=2U$ in \eqref{vor0}. Then
\begin{align*}
   S(q_1q_2x)
=&R_0(x; y)\!+\!R_{12}(x; y, H)\!+\!R_{21}(x; y, H)\!\\
\nonumber&+\!G_{12}(x; H)\!+\!G_{21}(x; H)
  \!+\!O\big(q_1q_2\log^3U\big).
\end{align*}
Take $H=U$, $y=\min\big(\frac{1}{2}Uh^{-1},U\log^{-6} U\big)$. From \cite[Lemma 6.2 and Lemma 6.5]{jia2016weighted}, we see
\begin{gather*}
\int_U^{2U}|G_{12}(x; H)\!+\!G_{21}(x; H)|^2dx\ll (q_1q_2)^2U\log U,\label{G2}\\
\int_U^{2U}|R_{12}(x; y, H)\!+\!R_{21}(x; y, H)|^2dx\ll (q_1q_2)^2U^\frac{3}{2}y^{-\frac{1}{2}}\log^3 U.\label{R2}
\end{gather*}
Thus we get
\begin{align}\label{s/r}
    \int_U^{2U}\!\!\!\big(S(q_1q_2x)-\!R_{0}(x; y)\big)^2dx\!\ll& (q_1q_2)^2U^\frac{3}{2}y^{-\frac{1}{2}}\log^3 U\!+\!(q_1q_2)^2U\log^6 U\\
   \nonumber \ll& (q_1q_2)^2Uh^\frac{1}{2}\log^3 U+(q_1q_2)^2U\log^6 U.
\end{align}

We now estimate the integral $\int_U^{2U}\big(R_{0}(x+h; y)-R_{0}(x; y)\big)^2dx$. From \eqref{R_0}, we have
\begin{equation}\label{dR0}
    R_{0}(x+h; y)-R_{0}(x; y)=F_1(x)+F_2(x),
\end{equation}
where
\begin{align*}
    F_1(x)=&\frac{q_1q_2}{4\sqrt{2}\pi}\big((x+h)^\frac{1}{4}\!-\!x^\frac{1}{4}\big)\sum_{n\leq y}\frac{\Delta d_2(n;a_1, q_1, a_2, q_2)}{n^{\frac{3}{4}}}\cos\big(4\pi\sqrt{n(x+h)}\!-\!\frac{3\pi}{4}\big),\\
    F_2(x)\!=&\frac{q_1q_2 x^\frac{1}{4}}{4\sqrt{2}\pi}\!\!\sum_{n\leq y}\!\frac{\Delta d_2(\!n;a_1, q_1, a_2, q_2\!)}{n^{\frac{3}{4}}}\!\Big(\!\!\cos\!\big(\!4\pi\sqrt{n(x\!+\!h)}\!-\!\!\frac{3\pi}{4}\!\big) \!-\!\cos\!\big(\!4\pi\sqrt{nx}\!-\!\!\frac{3\pi}{4}\!\big)\!\Big).
\end{align*}

From \cite[Lemma 6.3]{jia2016weighted}, we get
\begin{align}\label{F1}
   \int_U^{2U}F_1^2(x)dx\ll h^2U^{-2}\int_U^{2U}R_0^2(x+h)dx\ll (q_1q_2)^2h^2U^{-\frac{1}{2}}.
\end{align}

For  the mean square of $F_2(x)$, we see
\begin{equation}\label{F20}
    F_2^2=F_{21}+F_{22},
\end{equation}
with
\begin{align*}
    F_{21}(x)=&\frac{(q_1q_2)^2}{32\pi^2}x^\frac{1}{2}\sum_{n\leq y}\frac{\Delta d_2^2(n; a_1, q_1, a_2, q_2)}{n^{\frac{3}{2}}}\\
    &\times\Big(\cos\big(4\pi\sqrt{n(x+h)}-\frac{3\pi}{4}\big) -\cos\big(4\pi\sqrt{nx}-\frac{3\pi}{4}\big)\Big)^2,\\
    F_{22}(x)=&\frac{(q_1q_2)^2}{32\pi^2}x^\frac{1}{2}\sum_{\begin{subarray}{c}m,n\leq y\\m\neq n\end{subarray}}\frac{\Delta d_2(m; a_1, q_1, a_2, q_2)\Delta d_2(n;a_1, q_1, a_2, q_2)}{(mn)^{\frac{3}{4}}}\\
    &\times\Big(\cos\big(4\pi\sqrt{m(x+h)}-\frac{3\pi}{4}\big) -\cos\big(4\pi\sqrt{mx}-\frac{3\pi}{4}\big)\Big)\\
    &\times\Big(\cos\big(4\pi\sqrt{n(x+h)}-\frac{3\pi}{4}\big) -\cos\big(4\pi\sqrt{nx}-\frac{3\pi}{4}\big)\Big).
\end{align*}
By writing
 \begin{equation*}
    \cos\!\big(4\pi\sqrt{n(x\!+\!h)}-\frac{3\pi}{4}\!\big) -\cos\!\big(4\pi\sqrt{nx}-\frac{3\pi}{4}\big) \!\!=\!\!\sum_{j=0}^1(\!-\!1)^{j+1}\!\cos\!\big(4\pi\sqrt{n(x\!\!+\!\!jh)}-\frac{3\pi}{4}\big),
 \end{equation*}
 we get
 \begin{align}\label{F220}
F_{22}(x)\!=&\frac{(q_1q_2)^2}{32\pi^2}x^\frac{1}{2}\!\!\sum_{j_1=0}^1\sum_{j_2=0}^1\!(\!-\!1)^{j_1\!+\!j_2}\!\! \!\sum_{\begin{subarray}{c}m,n\leq y\\m\neq n\end{subarray}}\!\!\!\frac{\Delta d_2(m; a_1, q_1, a_2, q_2)\Delta d_2(n;a_1, q_1, a_2, q_2)}{(mn)^{\frac{3}{4}}}\\
    \nonumber &\times \cos\big(4\pi\sqrt{m(x\!+\!j_1h)}-\!\frac{3\pi}{4}\big)\cos\big(4\pi\sqrt{n(x\!+\!j_2h)}\!-\!\frac{3\pi}{4}\big)\\
   \nonumber =&:F_{221}(x)+F_{222}(x),
 \end{align}
where
\begin{align*}
    F_{221}(x) =&\frac{(q_1q_2)^2}{64\pi^2}x^\frac{1}{2} \sum_{j_1 =0}^1\sum_{j_2=0}^1 ( - 1)^{j_1 + j_2} \sum_{\begin{subarray}{c}m,n\leq y\\m\neq n\end{subarray}} \frac{1}{(mn)^{\frac{3}{4}}}\Delta d_2(m; a_1, q_1, a_2, q_2)\\
    \nonumber &\times  \Delta d_2(n;a_1, q_1, a_2, q_2)\cos\big(4\pi\sqrt{m(x\!+\!j_1h)}-\!4\pi\sqrt{n(x\!+\!j_2h)}\big),
 \end{align*}
\begin{align*}
    F_{222}(x)=&\frac{(q_1q_2)^2}{64\pi^2}x^\frac{1}{2}\!\sum_{j_1=0}^1\sum_{j_2=0}^1(-1)^{j_1+j_2+1}\!\!\sum_{\begin{subarray}{c}m,n\leq y\\m\neq n\end{subarray}}\!\!\frac{1}{(mn)^{\frac{3}{4}}}\Delta d_2(m; a_1, q_1, a_2, q_2)\\
    \nonumber &\times \Delta d_2(n;a_1, q_1, a_2, q_2)\sin\big(4\pi\sqrt{m(x\!+\!j_1h)}+\!4\pi\sqrt{n(x\!+\!j_2h)}\big).
\end{align*}
Let
\begin{equation*}
    g_\pm(x)=4\pi\sqrt{m(x\!+\!j_1h)}\pm\!4\pi\sqrt{n(x\!+\!j_2h)}.
\end{equation*}
Using
\begin{equation*}
    (1+t)^\frac{1}{2}=1+\sum_{v=1}^\infty d_vt^v\quad\big(|t|\leq\frac{1}{2}\big),
\end{equation*}
with $|d_v|<1$, we see
\begin{equation*}
    g_\pm(x)=4\pi\sqrt{x}(\sqrt{m}\pm\sqrt{n})+4\pi\sum_{v=1}^\infty \frac{d_vh^v}{x^{v-\frac{1}{2}}}(\sqrt{m}j_1^v\pm\sqrt{n}j_2^v).
\end{equation*}
Noting that $m,n\leq y\leq \frac{1}{2}Uh^{-1}$, we have
\begin{equation*}
    |g'_\pm(x)|\gg \frac{1}{\sqrt{x}}|\sqrt{m}\pm\sqrt{n}|\quad(m\neq n).
\end{equation*}
Then by the  the first derivative test we get
\begin{align*}
    \int_U^{2U}F_{221}(x)dx\ll&(q_1q_2)^2U\sum_{\begin{subarray}{c}m,n\leq y\\m\neq n\end{subarray}} \frac{\Delta d_2(m; a_1, q_1, a_2, q_2)\Delta d_2(n;a_1, q_1, a_2, q_2)}{(mn)^{\frac{3}{4}}|\sqrt{m}-\sqrt{n}|},\\
    \int_U^{2U}F_{222}(x)dx\ll&(q_1q_2)^2U\sum_{\begin{subarray}{c}m,n\leq y\\m\neq n\end{subarray}} \frac{\Delta d_2(m; a_1, q_1, a_2, q_2)\Delta d_2(n;a_1, q_1, a_2, q_2)}{(mn)^{\frac{3}{4}}|\sqrt{m}+\sqrt{n}|}.
\end{align*}
From \eqref{F220}, we obtain
\begin{align}\label{F22}
    \int_U^{2U}\!\!\!F_{22}(x)dx\!\ll&(q_1q_2)^2U\!\!\sum_{\begin{subarray}{c}m,n\leq y\\m\neq n\end{subarray}} \!\!\frac{\Delta d_2(m; a_1, q_1, a_2, q_2)\Delta d_2(n;a_1, q_1, a_2, q_2)}{(mn)^{\frac{3}{4}}|\sqrt{m}-\sqrt{n}|}\\
   \nonumber \ll& (q_1q_2)^2U\log^4y,
\end{align}
where we used the estimate  $\sum_{n\leq N}d(n)\ll N\log N$.

By the  elementary formulas
\begin{align*}
    \cos u-\cos v=-2\sin\big(\frac{u+v}{2}\big)\sin\big(\frac{u-v}{2}\big),\quad\text{and}\quad \sin^2u=\frac{1}{2}(1-\cos 2u),
\end{align*}
we have
\begin{align}\label{F210}
   &\int_U^{2U}F_{21}(x)dx\\
  \nonumber =&\frac{(q_1q_2)^2}{8\pi^2}\!\sum_{n\leq y}\!\frac{\Delta d_2^2(n; a_1, q_1, a_2, q_2)}{n^{\frac{3}{2}}}\!\int_u^{2U}\!\!\!\!x^\frac{1}{2}\\
   \nonumber &\times \sin^2\!\big(2\pi\sqrt{n(x\!+\!h)} \!+\!2\pi\sqrt{nx}\!-\!\frac{3\pi}{4}\big)\sin^2\!\big(2\pi\sqrt{n(x\!+\!h)}\!-\!2\pi\sqrt{nx}\big)dx\\
   \nonumber =&:I_{211}+I_{212},
\end{align}
where
\begin{align*}
    I_{211}=&\frac{(q_1q_2)^2}{16\pi^2}\!\sum_{n\leq y}\!\frac{\Delta d_2^2(n; a_1, q_1, a_2, q_2)}{n^{\frac{3}{2}}}\!\!\int_U^{2U}\!\!\!x^\frac{1}{2} \sin^2\big(2\pi\sqrt{n(x\!+\!h)}\!-\!2\pi\sqrt{nx}\big)dx,\\
    I_{212}=&\frac{(q_1q_2)^2}{16\pi^2}\!\sum_{n\leq y}\!\frac{\Delta d_2^2(n; a_1, q_1, a_2, q_2)}{n^{\frac{3}{2}}}\!\!\int_U^{2U}\!\!\!x^\frac{1}{2} \sin\big(4\pi\sqrt{n(x\!+\!h)}\!+\!4\pi\sqrt{nx}\big)\\
    &\hspace{10em}\times \sin^2\big(2\pi\sqrt{n(x\!+\!h)}\!-\!2\pi\sqrt{nx}\big)dx.
\end{align*}
 By the first derivative test, we have
 \begin{equation*}
    L_n(t):=\int_U^tx^\frac{1}{2} \sin\big(4\pi\sqrt{n(x\!+\!h)}\!+\!4\pi\sqrt{nx}\big)dx \ll Un^{-\frac{1}{2}},\quad U\leq t\leq2U.
 \end{equation*}
Using the integration by parts, we obtain
\begin{align*}
    &\int_U^{2U}\!\!\!x^\frac{1}{2} \sin\big(4\pi\sqrt{n(x\!+\!h)}\!+\!4\pi\sqrt{nx}\big)\sin^2\big(2\pi\sqrt{n(x\!+\!h)}\!-\!2\pi\sqrt{nx}\big)dx\\
    =&\int_U^{2U}\!\!\!\sin^2\big(2\pi\sqrt{n(x\!+\!h)}\!-\!2\pi\sqrt{nx}\big)dL_n(x)\\
    =&L_n(2U)\sin^2\!\big(2\pi\sqrt{n(2U\!+\!h)}\!-\!2\pi\sqrt{2nU}\big)\!- \! 2\int_U^{2U}\!\!\!L_n(x)\Big(\frac{\pi\sqrt{n}}{\sqrt{x\!+\!h}} \!-\!\frac{\pi\sqrt{n}}{\sqrt{x}}\Big)\\
    &\!\times\!\sin\big(2\pi\sqrt{n(x\!+\!h)}\!-\!2\pi\sqrt{nx}\big)\cos\big(2\pi\sqrt{n(x\!+\!h)}\!-\!2\pi\sqrt{nx}\big) dx\\
    \ll&Un^{-\frac{1}{2}}+U^{\frac{1}{2}}h,
\end{align*}
which yields
\begin{align}\label{I212}
    I_{212}\ll&(q_1q_2)^2\sum_{n\leq y}\!\frac{d^2(n)}{n^{\frac{3}{2}}}\big(Un^{-\frac{1}{2}}+U^{\frac{1}{2}}h\big)\\  \nonumber\ll&(q_1q_2)^2\big(U+U^{\frac{1}{2}}h\big)\ll(q_1q_2)^2U.
\end{align}
By using
\begin{equation*}
    \sqrt{x+h}=x^{\frac{1}{2}}+hx^{-\frac{1}{2}}+O(h^2x^{-\frac{3}{2}}),\quad x\geq100h^2,
\end{equation*}
we get
\begin{align*}
    \sin^2\big(2\pi\sqrt{n(x\!+\!h)}\!-\!2\pi\sqrt{nx}\big)=&\sin^2\big(\pi hn^{\frac{1}{2}}x^{-\frac{1}{2}}+O(h^2n^\frac{1}{2}x^{-\frac{3}{2}})\big)\\
    =&\sin^2\big(\pi hn^{\frac{1}{2}}x^{-\frac{1}{2}}\big) +O(h^2n^\frac{1}{2}x^{-\frac{3}{2}}).
\end{align*}
Noting that
\begin{align*}
    \int_U^{2U}\!\!\!x^\frac{1}{2}\sin^2\big(\pi hn^{\frac{1}{2}}x^{-\frac{1}{2}}\big)dx
    \ll& \int_U^{2U}\!\!\!x^\frac{1}{2}\min\big(1,h^2nx^{-1}\big)dx\\
    \ll&\left\{\begin{array}{ll}U^\frac{1}{2}h^2n,&n\leq Uh^{-2},\\U^{\frac{3}{2}},&n> Uh^{-2},\end{array}\right.
\end{align*}
we have
\begin{align}\label{I211}
    &I_{211}\ll(q_1q_2)^2\!\sum_{n\leq y}\!\frac{ d^2(n)}{n^{\frac{3}{2}}}\!\!\int_U^{2U}\!\!\!x^\frac{1}{2}\big(\sin^2(\pi hn^{\frac{1}{2}}x^{-\frac{1}{2}}) +O(h^2n^\frac{1}{2}x^{-\frac{3}{2}})\big) dx\\
   \nonumber \ll&(q_1q_2)^2\!\sum_{n\leq y}\!\frac{ d^2(n)}{n^{\frac{3}{2}}}\!\!\int_U^{2U}\!\!\!x^\frac{1}{2}\sin^2(\pi hn^{\frac{1}{2}}x^{-\frac{1}{2}}) dx\!+\!O\Big((q_1q_2 h)^2\!\sum_{n\leq y}\!\frac{ d^2(n)}{n}\Big)\\
   \nonumber\ll& (q_1q_2 h)^2 U^\frac{1}{2}\!\!\!\sum_{n\leq Uh^{-2}}\!\!\!\frac{ d^2(n)}{n^\frac{1}{2}} \!+\!(q_1q_2)^2 U^\frac{3}{2}\!\!\!\sum_{n> Uh^{-2}}\!\!\frac{ d^2(n)}{n^\frac{3}{2}} \!+\!O\big((q_1q_2 h)^2\log^4y\big)\\
   \nonumber\ll& (q_1q_2)^2U h\log^3\frac{ \sqrt{U}}{h},
\end{align}
where we used the well-known estimate $\sum_{n\leq N}d^2(n)\ll N\log^3N$.

From \eqref{F210}-\eqref{I211}, we get
\begin{equation}\label{F21}
    \int_U^{2U}F_{21}(x)dx\ll(q_1q_2)^2U h\log^3\frac{ \sqrt{U}}{h}.
\end{equation}

Combining \eqref{F20}, \eqref{F22} and \eqref{F21}, we obtain
\begin{equation*}
    \int_U^{2U}F_{2}^2(x)dx\ll(q_1q_2)^2U h\log^3\frac{\sqrt{U}}{h}+(q_1q_2)^2U\log^4y,
\end{equation*}
which together with \eqref{dR0}, \eqref{F1} yields
\begin{equation}\label{dR}
    \!\int_U^{2U}\!\!\!\!\!\big(R_{0}(x\!+\!h; y)\!-\!R_{0}(x; y)\big)^2\!dx\!\ll\!(q_1q_2)^2U h\log^3\frac{\sqrt{U}}{h}\!+\!(q_1q_2)^2U\log^4y.
\end{equation}

From \eqref{J}, \eqref{s/r}, and \eqref{dR}, it follows that
\begin{equation*}
    J(U,h)\ll (q_1q_2)^2U h\log^3\frac{\sqrt{U}}{h}+(q_1q_2)^2U\log^6y,
\end{equation*}
which implies
\begin{equation}\label{int2}
    \int_2\ll (q_1q_2)^2T h\log^3\frac{\sqrt{T}}{h}+(q_1q_2)^2T\mathcal{L}^6,
\end{equation}
via a splitting argument. Then Lemma \ref{lem:meanvalue} follows from \eqref{I}, \eqref{int1} and \eqref{int2}.
\end{proof}
\vspace{2ex}

\section{Proof of Theorem \ref{th:maintain}}

In this section, we will give a proof of Theorem \ref{th:maintain} by following the approach of \cite{tsang2012sign}. We still write $S(q_1q_2x)=S\big(q_1q_2x; \frac{a_1}{q_1}, \frac{a_2}{q_2}\big)$. Define
\begin{align*}
    S_+(t)=&\frac{1}{2}\big(|S(t)|+S(t)\big),&
    S_-(t)=&\frac{1}{2}\big(|S(t)|-S(t)\big).
\end{align*}
We need the following two lemmas.
\begin{lemma}\label{lem:3}
 \begin{equation*}
    \int_T^{2T}S^2_\pm(q_1q_2t)dt\gg (q_1q_2)^2T^\frac{3}{2}.
 \end{equation*}
\end{lemma}
\begin{proof}
From \eqref{sk} with $k=2,4$, by H\"{o}lder's inequality, we get
\begin{align*}
    (q_1q_2)^2T^\frac{3}{2}\ll\int_T^{2T}S^2(q_1q_2t)dt \ll&\Big(\int_T^{2T}|S(q_1q_2t)|dt\Big)^\frac{2}{3}\Big(\int_T^{2T}S^4(q_1q_2t)dt\Big)^\frac{1}{3}\\
    \ll&\Big(\int_T^{2T}|S(q_1q_2t)|dt\Big)^\frac{2}{3}(q_1q_2)^\frac{4}{3}T^\frac{2}{3},
\end{align*}
which yields
\begin{equation}\label{31}
    \int_T^{2T}|S(q_1q_2t)|dt\gg q_1q_2T^\frac{5}{4}.
\end{equation}
From \eqref{s1}, we see
\begin{equation*}
        \int_T^{2T}S(q_1q_2t)dt\ll q_1q_2T^\frac{3}{4}.
\end{equation*}
Thus, from the definition of $S_\pm(q_1q_2t)$, we have
\begin{equation*}
    \int_T^{2T}S_\pm(q_1q_2t)dt\gg q_1q_2T^\frac{5}{4}.
\end{equation*}
Then by  Cauchy-Schwarz's inequality, we get
\begin{align*}
    q_1q_2T^\frac{5}{4}\ll\Big(\int_T^{2T}dt\Big)^\frac{1}{2} \Big(\int_T^{2T}S^2_\pm(q_1q_2t)dt\Big)^\frac{1}{2}\ll T^\frac{1}{2}\Big(\int_T^{2T}S^2_\pm(q_1q_2t)dt\Big)^\frac{1}{2},
\end{align*}
which immediately implies Lemma \ref{lem:3}.
\end{proof}

\begin{lemma}\label{lem:4}
Suppose $2\leq H_0\leq\sqrt{T}$. Then
\begin{equation*}
    \int_T^{2T}\max_{h\leq H_0}\big(S_\pm(q_1q_2(t+h))-S_\pm(q_1q_2t)\big)^2 dt\ll (q_1q_2)^2H_0T\mathcal{L}^7.
\end{equation*}
\end{lemma}

\begin{proof}
Since
\begin{equation*}
    |S_\pm(q_1q_2(t+h))-S_\pm(q_1q_2t)|\leq |S(q_1q_2(t+h))-S(q_1q_2t)|,
\end{equation*}
it is sufficient to prove that
\begin{equation*}
   I= \int_T^{2T}\max_{h\leq H_0}\big(S(q_1q_2(t+h))-S(q_1q_2t)\big)^2 dt\ll (q_1q_2)^2H_0T\mathcal{L}^7.
\end{equation*}

Write $H_0=2^\lambda b$, such that $\lambda\in \mathbb{N}$ and $1\leq b <2$. By using Lemma \ref{lem:meanvalue}, we get
\begin{align*}
    I\!\ll& \lambda\sum_{\mu\leq\lambda}\sum_{0\leq\nu\leq2^\mu} \int_{T+\nu2^{\lambda-\mu}b}^{2T+\nu2^{\lambda-\mu}b} \big(S(q_1q_2(t+2^{\lambda-\mu}b))-S(q_1q_2t)\big)^2dt\!+\!(q_1q_2)^2T\mathcal{L}^2\\
    \ll& \lambda\sum_{\mu\leq\lambda}\sum_{0\leq\nu\leq2^\mu} \big((q_1q_2)^22^{\lambda-\mu}bT\mathcal{L}^3+(q_1q_2)^2T\mathcal{L}^6\big)\\
    \ll& \lambda\sum_{\mu\leq\lambda}\big((q_1q_2)^22^{\lambda}bT\mathcal{L}^3+(q_1q_2)^22^\mu T\mathcal{L}^6\big)\\
    \ll& \lambda^2(q_1q_2)^2H_0T\mathcal{L}^3+\lambda(q_1q_2)^2H_0T\mathcal{L}^6\\
    \ll& (q_1q_2)^2H_0T\mathcal{L}^7,
\end{align*}
where we used the well-known estimate
\begin{equation*}
    \sum_{x<n\leq x+y}d(n)\ll y\log x,\quad x^\varepsilon<y<x.
\end{equation*}
\end{proof}
Now we finish the proof of Theorem \ref{th:maintain}. For any function $P(t)$ and $Q(t)$ such that
\begin{align*}
    \omega(t)=P^2(t)-4\max_{h\leq H_0}\big(P(t+h)-P(t)\big)^2-Q^2(t)>0,
\end{align*}
we see that $P(t+h)$  has the same sign as $P(t)$, and $|P(t+h)|>\frac{1}{2}|Q(t)|$ for any $0\leq h\leq H_0$. Take $P(t)=S_\pm(q_1q_2t)$ and $Q(t)=\delta q_1q_2t^\frac{1}{4}$ for a  sufficiently small $\delta>0$. By Lemma \ref{lem:3} and Lemma \ref{lem:4}, we get
\begin{align}\label{35}
    \int_T^{2T}\!\!\!\!\!\omega(t)dt\!\gg\! (q_1q_2)^2T^\frac{3}{2}\!-\!O\Big((q_1q_2)^2\big(H_0T\mathcal{L}^7+\delta^2T^\frac{3}{2}\big)\Big)\!\gg\! (q_1q_2)^2T^\frac{3}{2},
\end{align}
by taking $H_0=\delta T^\frac{1}{2}\mathcal{L}^{-7}$. Let $$\mathscr{S}=\{t\in[T,2T]:\omega(t)>0\}.$$ From \eqref{sk} and \eqref{35}, using Cauchy-Schwarz's inequality, we have
\begin{align*}
    (q_1q_2)^2T^\frac{3}{2}\ll & \int_T^{2T}\omega(t)dt\leq \int_\mathscr{S}\omega(t)dt \leq \int_\mathscr{S}S^2_\pm(q_1q_2t)dt \\ \leq&|\mathscr{S}|^\frac{1}{2}\Big(\int_T^{2T}S^4(q_1q_2t)dt\Big)^\frac{1}{2} \ll |\mathscr{S}|^\frac{1}{2}(q_1q_2)^2T,
\end{align*}
which implies
\begin{equation*}
    |\mathscr{S}|\gg T.
\end{equation*}
Thus the proof of Theorem \ref{th:maintain} is completed.
\qed

\section{Proof of Theorem \ref{th:omega}}

Suppose $k\geq3$ is a fixed odd integer and $T\geq (q_1q_2)^{\varepsilon}$ is a large parameter.
Set
\begin{equation*}
    \delta=\left\{\begin{array}{ll}-1,&\quad\text{if } C_k\geq0,\\1,&\quad\text{if } C_k<0,\end{array}\right.
\end{equation*}
where $C_k$ is defined in \eqref{sk}.

By Theorem \ref{th:maintain}, there exists $t\in[T,2T]$ such that $\delta S\big(q_1q_2u; \frac{a_1}{q_1}, \frac{a_2}{q_2}\big)> c_5 q_1q_2t^\frac{1}{4}$ for any $u\in[t,t+H_0]$, with $H_0=c_4\sqrt{T}\mathcal{L}^{-7}$. Thus
\begin{align*}
    &c^k_5 H_0 t^\frac{k}{4}<(q_1q_2)^{-k}\int_t^{t+H_0}\delta^kS^k\big(q_1q_2u; \frac{a_1}{q_1}, \frac{a_2}{q_2}\big)du\\
    =&\delta^kC_k\big((t\!+\!H_0)^{1+\frac{k}{4}}\!-\!t^{1+\frac{k}{4}}\big)\!+\delta^k\Big( \mathcal{F}_k\big(q_1q_2(t\!+\!H_0); \frac{a_1}{q_1}, \frac{a_2}{q_2}\big)\!-\! \mathcal{F}_k\big(q_1q_2t; \frac{a_1}{q_1}, \frac{a_2}{q_2}\big)\Big)\\
    =&\delta^kC_k(1\!\!+\!\frac{k}{4})t^{\frac{k}{4}}H_0\!\!+\!O(H_0^2t^{\frac{k}{4}-1})\!+\!\delta^k\big( \!\mathcal{F}_k(q_1q_2(t\!\!+\!\!H_0); \frac{a_1}{q_1}, \frac{a_2}{q_2})\!\!-\!\!\mathcal{F}_k(q_1q_2t; \frac{a_1}{q_1}, \frac{a_2}{q_2})\big),
\end{align*}
which yields
\begin{equation*}
    \delta^k\Big( \!\mathcal{F}_k\big(q_1q_2(t\!+\!H_0); \frac{a_1}{q_1}, \frac{a_2}{q_2}\big)\!\!- \! \mathcal{F}_k\big(q_1q_2t; \frac{a_1}{q_1}, \frac{a_2}{q_2}\big)\Big)>C_k^*H_0t^{\frac{k}{4}}\big(1+O(H_0T^{-1})\big),
\end{equation*}
with
\begin{equation*}
    C_k^*=c^k_5-\delta^kC_k\big(1\!+\!\frac{k}{4}\big)>0.
\end{equation*}
Thus we get
\begin{equation*}
    \Big|\mathcal{F}_k\big(q_1q_2(t\!+\!H_0); \frac{a_1}{q_1}, \frac{a_2}{q_2}\big)\!\!- \! \mathcal{F}_k\big(q_1q_2t; \frac{a_1}{q_1}, \frac{a_2}{q_2}\big)\Big|\gg H_0T^{\frac{k}{4}},
\end{equation*}
which immediatly implies Theorem \ref{th:omega}.
\qed

\vspace{5ex}


\begin{thebibliography}{99}
\label{chapter5}
\bibitem{Berndt12} B. C. Berndt, S. Kim, and A. Zaharescu. Weighted divisor sums and Bessel function series,
IV. The Ramanujan Journal, 29:79-102, 2012.
\bibitem{Berndt06} B. C. Berndt and A. Zaharescu. Weighted divisor sums and Bessel function series. Mathema-
tische Annalen, 335(2):249-283, 2006.
\bibitem{cramer1922zwei} H. Cram{\'e}r. {\"U}ber zwei S{\"a}tze des Herrn G. H. Hardy. Mathematische Zeitschrift, 15(1):201-210,
1922.
\bibitem{good1977einomega} A. Good. Ein ω-resultat für das quadratische mittel der riemannschen zetafunktion auf der
kritischen linie. Inventiones mathematicae, 41(3):233-251, 1977.
\bibitem{heath1992distribution} D. R. Heath-Brown. The distribution and moments of the error term in the dirichlet divisor
problem. Acta Arith, 60(4):389-415, 1992.
\bibitem{heathbrown1994sign} D. R. Heathbrown and K. Tsang. Sign changes of e (t), δ (x), and p (x). Journal of Number
Theory, 49(1):73-83, 1994.
\bibitem{Huxley03} M. N. Huxley. Exponential sums and lattice points III. Proceedings of the London Mathe-
matical Society, 87(03):591-609, 2003.
\bibitem{huxley2005exponential} M. N. Huxley. Exponential sums and the Riemann zeta function V. Proceedings of the London
Mathematical Society, 90(01):1-41, 2005.
\bibitem{ivic1983large} A. Ivi{\'c}. Large values of the error term in the divisor problem. Inventiones mathematicae,
71(3):513-520, 1983.
\bibitem{ivic2007higher} A. Ivi′ c and P. Sargos. On the higher moments of the error term in the divisor problem. Illinois
Journal of Mathematics, 51(2):353-377, 2007.
\bibitem{jia2016weighted} L. R. Jia and W. G. Zhai. A weighted divisor problem. arXiv preprint arXiv:1602.06160,
2016.
 \bibitem{jutila1984divisor}M. Jutila. On the divisor problem for short intervals. Ann. Univer. Turkuensis Ser. AI,
186:23-30, 1984.
\bibitem{tsang1992higher} K. M. Tsang. Higher-power moments of δ(x), e(t) and p(x). Proceedings of the London Math-
ematical Society, 65:65-84, 1992.
\bibitem{tsang2012sign}  K. M. Tsang and W. G. Zhai. Sign changes of the error term in weyls law for heisenberg
manifolds. Transactions of the American Mathematical Society, 364(5):2647-2666, 2012.
\bibitem{voronoi1904fonction} G. Vorono{\"\i}. Sur une fonction transcendante et ses applications {\`a} la sommation de quelques s{\'e}ries. In Annales scientifiques de l'{\'E}cole Normale Sup{\'e}rieure, volume 21, pages 207-267,
1904.
\bibitem{zhai04} W. G. Zhai. On higher-power moments of $\Delta(x)$(II). Acta Arith, 114:35-54, 2004.
\end{thebibliography}
\end{document}